\documentclass{amsart}
\usepackage{amscd,amsmath,amssymb,amsthm,bm,cases,color,comment,mathrsfs,mathtools}
\usepackage[dvips]{graphicx}
\usepackage[all]{xy}
\makeatletter
\@namedef{subjclassname@2020}{%
\textup{2020} Mathematics Subject Classification}
\makeatother
\theoremstyle{definition}
\newtheorem{definition}{Definition}[section]
\newtheorem{remark}[definition]{Remark}

\theoremstyle{plain}
\newtheorem{theorem}[definition]{Theorem}
\newtheorem{proposition}[definition]{Proposition}
\newtheorem{lemma}[definition]{Lemma}
\newtheorem{corollary}[definition]{Corollary}
\numberwithin{equation}{section}
\title[Distortion between mapping class groups]{The mapping class group of a nonorientable
surface is quasi-isometrically embedded in the mapping class group of the orientation double
cover}
\author[T.~Katayama]{Takuya Katayama}
\address{
(Takuya Katayama)
Department of Mathematics,
Faculty of Science,
Gakushuin University,
1-5-1 Mejiro, Toshima-ku, Tokyo 171-8588, Japan
}
\email{katayama@math.gakushuin.ac.jp}
\author[E.~Kuno]{Erika Kuno}
\address{
(Erika Kuno)
Department of Mathematics,
Graduate School of Science,
Osaka University,
1-1 Machikaneyama-cho Toyonaka, Osaka 560-0043, Japan
}
\email{e-kuno@math.sci.osaka-u.ac.jp}
\date{\today}
\begin{document}
\begin{abstract}
Let $N$ be a connected nonorientable surface with or without boundary and punctures, and
$j\colon S\rightarrow N$ be the orientation double covering. 
It has previously been proved that the orientation double covering $j$ induces an embedding $\iota\colon\mathrm{Mod}(N)$ $\hookrightarrow$ $\mathrm{Mod}(S)$ with one exception. 
In this paper, we prove that this injective homomorphism $\iota$ is a quasi-isometric embedding. 
The proof is based on the semihyperbolicity of $\mathrm{Mod}(S)$, which has already been
established. 
We also prove that the embedding $\mathrm{Mod}(F') \hookrightarrow \mathrm{Mod}(F)$
induced by an inclusion of a pair of possibly nonorientable surfaces $F' \subset F$ is a quasi-isometric embedding. 
\end{abstract}
\maketitle
\section{Introduction}\label{Introduction}
Let $S=S_{g,p}^{b}$ be the compact connected orientable surface of genus $g$ with $b$
boundary components and $p$ punctures, and $N=N_{g,p}^{b}$ be the compact connected
nonorientable surface of genus $g$ with $b$ boundary components and $p$ punctures.
In the case where $b=0$ or $p=0$, we drop the suffix that denotes $0$, excepting $g$, from
$S_{g, p}^{b}$ and $N_{g, p}^{b}$. 
For example, $N_{g, 0}^{0}$ is simply denoted as $N_{g}$. 
If we are not interested in whether a given surface is orientable or not, we denote the surface by $F$. 
The {\it mapping class group} $\mathrm{Mod}(F)$ of $F$ is the group of isotopy classes of
homeomorphisms on $F$ which are orientation-preserving if $F$ is orientable and preserve
$\partial F$ pointwise. 
Recall that if $H \subset G$ is a pair of finitely generated groups with word metrics $d_H$ and $d_G$ (induced by finite generating sets), respectively, then the {\it distortion} of $H$ in $G$ is defined as
$$ \delta_{H}^{G}(n) := \mathrm{max} \{ d_H(1, h) \mid h \in H \ \mbox{with} \ d_{G}(1,
h) \leq n \} .$$ 
This function is independent of the choice of word metrics $d_H$ and $d_G$ up to Lipschitz
equivalence. 
In addition, there exists a constant $K$ such that $\delta_{H}^{G}(n) \leq Kn$ if and only if the inclusion $H \subset G$ is a quasi-isometric embedding. 
The subgroup $H$ is said to be {\it undistorted} in $G$ if this condition is satisfied;
otherwise, we say that $H$ is {\it distorted}. 
The distortions of various subgroups in the mapping class groups of orientable surfaces have been extensively investigated. 
For example, the mapping class groups of subsurfaces are undistorted according to Masur--
Minsky~\cite[Theorem 6.12]{Masur-Minsky00} and Hamenst\"{a}dt~\cite[Proposition
4.1]{Hamenstad09b}. 
Farb--Lubotzky--Minsky~\cite{Farb-Lubotzky-Minsky01} proved that groups generated by
Dehn twists about disjoint curves are undistorted. 
Moreover, Rafi--Schleimer~\cite{Rafi-Schleimer09} proved that an orbifold covering map of
orientable surfaces induces a quasi-isometric embedding between the mapping class groups. 
For examples of distorted subgroups of mapping class groups, see Broaddus--Farb--Putman~\cite{Broaddus-Farb-Putman11}, Cohen~\cite{Cohen14}, and Kuno--Omori~\cite{Kuno-Omori17}, where it is proved that the Torelli group $\mathcal{I}_{g}^{b}$ is distorted in $\mathrm{Mod}(S_{g}^{b})$. 
Moreover, it has been proved by Hamenst\"{a}dt--Hensel~\cite{Hamenstad-Hensel12} that
the handlebody group is exponentially distorted in the mapping class group of the boundary
surface. 

Birman--Chillingworth~\cite[Theorem 1]{Birman-Chillingworth72} (for closed surfaces),
Szepietowski~\cite[Lemma 3]{Szepietowski10} (for surfaces with boundaries), and
Gon\c{c}alves--Guaschi--Maldonado~\cite[Theorem 1.1]{Goncalves-Guaschi-Maldonado18}
(for surfaces with punctures) proved that the mapping class group of a nonorientable surface
$N_{g, p}^{b}$ is a subgroup of the mapping class group of the orientation double cover
$S_{g-1, 2p}^{2b}$ (we will describe the induced injective homomorphism in
Section~\ref{Nonorientable surface mapping class groups are undistorted}). 
In this paper, we prove the following theorem using the semihyperbolicity of the mapping
class group of orientable surfaces, independently established by Durham--Minsky--Sisto~
\cite[Corollary D]{Durham-Minsky-Sisto20} and Haettel--Hoda--Petyt~\cite[Corollary
3.11]{Haettel-Hoda-Petyt20}. 

\begin{theorem}\label{first_thm}
For all but $(g, p, b) = (2, 0, 0)$, the mapping class group $\mathrm{Mod}(N_{g,p}^{b})$ is undistorted in the mapping class group $\mathrm{Mod}(S_{g-1,2p}^{2b})$. 
\end{theorem}

We note that $\mathrm{Mod}(N_2)$ cannot be embedded in $\mathrm{Mod}(S_1)$ (see Remark~\ref{Klein_bottle_case}). 

The remainder of this paper is organized as follows. 
In Section~\ref{Preliminaries}, we first review the definition of ``semihyperbolic" groups, which captures the feature of CAT(0) groups in purely combinatorial group-theoretic terms. 
We then prepare Lemma~\ref{centralizer_of_mcg_is_undistorted} on the mapping class
groups of orientable surfaces, which is used to prove Theorem~\ref{first_thm}. 
Section~\ref{Nonorientable surface mapping class groups are undistorted} is devoted to proving Theorem~\ref{first_thm}. 
In this section, we also obtain the result that the mapping class groups of nonorientable
surfaces are semihyperbolic. 
Finally, in the appendix, we show that the injective homomorphism between the mapping
class groups, which comes from the inclusion of surfaces, is a quasi-isometric embedding
(Proposition~\ref{qie_nonori_inc}). 

{\bf Acknowledgements:} The authors wish to express their great appreciation to B\l a\.{z}ej Szepietowski for encouraging the second author to decide whether orientation double
coverings induce quasi-isometric embeddings when she visited his office in 2017. 
Moreover, he pointed out a few mistakes in the proofs of the main results and introduced the authors to Lemma~\ref{injective_homomorphism_induced_by_orientation_double_cover}. 
His suggestions have considerably improved the proofs and the main results of this paper. 
The authors express their gratitude to Harry Petyt for variable comments on the semihyperbolicity of extended mapping class groups. 
The authors are also deeply grateful to Martin Bridson and Saul Schleimer for answering their questions, and to Makoto Sakuma for his comments on an early version of this paper. 
The first author was supported by JSPS KAKENHI through grant number 20J1431, and the
second author was supported by JST, ACT-X, through grant number JPMJAX200D. 

\section{Preliminaries}\label{Preliminaries}
In this section, we show that the centralizer of every element in the (extended) mapping class group of an orientable surface is quasi-isometrically embedded in the (extended) mapping class group. 
Following Alonso--Bridson~\cite{Alonso-Bridson95}, we recall the definition of
``semihyperbolicity" for finitely generated groups. 
Throughout this section, we assume that $G$ is a finitely generated group and $X$ is a finite generating set of $G$. 
We write $X^{-1}=\{x^{-1}|x\in X\}$ for a set of formal inverses of $X$.
Set $\mathcal{A}\coloneqq X\sqcup X^{-1}$. 
We consider the free monoid $\mathcal{A}^{*}$ which consists of all finite words on
$\mathcal{A}$. 
The free group over $\mathcal{A}$ is naturally contained in $\mathcal{A}^{*}$, and the
natural projection $\mu \colon \mathcal{A}^{*}\rightarrow G$ is well-defined.
We denote the length of a word $w\in\mathcal{A}^{*}$ by $\ell(w)$.
For $g\in G$, we write 
\begin{equation*}
\|g\|_{(G,X)}=\mathrm{min}\{\ell(w)|w\in\mathcal{A}^{*}, \mu(w)=g\}.
\end{equation*}
Let $\mathrm{Cay}(G, X)$ be the Cayley graph of $G$ with respect to a generating set $X$.
Each vertex of $\mathrm{Cay}(G,X)$ corresponds to an element $g\in G$.
An edge of $\mathrm{Cay}(G,X)$ corresponds to $\mu(x)$ for some $x\in X$, and it is
oriented from $g$ to $g\mu(x)$ for each $g\in G$. 
We consider $\mathrm{Cay}(G, X)$ as a metric space by assigning a length of 1 to each edge. 
For $g,h\in G$, the distance in $\mathrm{Cay}(G, X)$ is $d(g,h)=\|g^{-1}h\|_{(G,X)}$. 
By this definition of the distance, the group $G$ acts by isometry on $\mathrm{Cay}(G,X)$ via the action $(g, x) \mapsto gx$. 
Let $\mathcal{P}(\mathrm{Cay}(G,X))$ be the set of all discrete paths which connect a pair
of vertices in $\mathrm{Cay}(G,X)$, and let $e \colon \mathcal{P}(\mathrm{Cay}(G,X))
\rightarrow \mathrm{Cay}(G,X) \times \mathrm{Cay}(G,X)$ be the map which sends a
discrete path $p$ of length $N$ to $(p(0), p(N))$. 

\begin{definition}
A {\it bicombing} $\sigma \colon \mathrm{Cay}(G,X) \times \mathrm{Cay}(G,X)
\rightarrow \mathcal{P}(\mathrm{Cay}(G,X))$ is a set-theoretic section for $e$. 
Namely, $\sigma$ satisfies $e \circ \sigma = \mathrm{id}$. 
A bicombing $\sigma$ is said to be {\it quasi-geodesic} if there exist $\lambda, 
\varepsilon \geq 0$ such that the image of $\sigma$ consists only of $(\lambda, \varepsilon)$-quasi-geodesics. 
Additionally, a bicombing $\sigma$ is said to be {\it bounded} if there exist $k_1 \geq 1, k_2 \geq 0$ such that, for all $x, y, x',y' \in \mathrm{Cay}(G,X)$ and $t \in \mathbb{N}$,
$$ d_{G}(\sigma(x, y)(t) , \sigma(x', y')(t)) \leq k_1 \mathrm{max} \{ d_{G}(x, y), d_{G}(x', y') \} + k_2 .$$
We say that a finitely generated group $G$ is {\it weakly semihyperbolic} if $\mathrm{Cay}(G,X)$ admits a bounded quasi-geodesic bicombing. 
For elements $g, h \in G$, the image $\sigma(g, h)$ is called a {\it combing line} from $g$ to $h$. 
In addition, the combing line $\sigma(1, g)$ is denoted simply by $\sigma(g)$. 
\end{definition}

Alonso--Bridson proved~\cite{Alonso-Bridson95} that being a weakly semihyperbolic group
is a quasi-isometric invariant for finitely generated groups.
Therefore, the definition of weak semihyperbolicity for finitely generated groups is
independent of the choice of a generating set. 

\begin{definition}
A bicombing $\sigma$ for a finitely generated group $G$ is said to be {\it equivariant} if
$\sigma$ satisfies the following: for all $g \in G$ and $x, y \in \mathrm{Cay}(G,X)$,
$$ g \cdot \sigma(x, y) = \sigma(gx, gy). $$
Here, $g \cdot \sigma(x, y)$ is the image of the discrete path $\sigma(x, y)$ under the left action of $G$ on $\mathrm{Cay}(G,X)$ and $gx, gy$ are multiplications on the group $G$. 
A quasi-geodesic, bounded, and equivariant bicombing for a finitely generated group is called a {\it semihyperbolic structure}. 
A finitely generated group is said to be {\it semihyperbolic} if it admits a semihyperbolic structure.
\end{definition}

Though the authors do not know whether semihyperbolicity for finitely generated groups is a quasi-isometric invariant, the definition of semihyperbolicity is known to be independent of the choice of a generating set \cite[Corollary 4.2]{Alonso-Bridson95}. 

\begin{definition}
Let $G$ be a group with a bicombing $\sigma$. 
Then, a subgroup $H \leq G$ is said to be $\sigma$-quasi-convex if there exists a constant $k\geq 0$ such that $d(\sigma(h)(t), H)\leq k$ for all $h\in H$ and $t\in\mathbb{N}$, where $\sigma(h)$ is the combing line from $1$ to $h$. 
\end{definition}

An important example of a $\sigma$-quasi-convex subgroup is a finite-index subgroup.
A finite-index subgroup is $\sigma$-quasi-convex because it is dense in the entire group.
This fact will be frequently used throughout this paper.

In a semihyperbolic group, every $\sigma$-quasi-convex subgroup is quasi-isometrically
embedded and inherits the semihyperbolicity. 

\begin{lemma}{\rm(}{\cite[Lemma 7.2 and Theorem 7.3]{Alonso-Bridson95}}{\rm)}\label{qconvex_qiemb}
Let $G$ be a finitely generated group with a bicombing $\sigma$ and let $H$ be a $\sigma$-
quasi-convex subgroup of $G$.
Then, the following holds.
\begin{enumerate}
 \item[(1)] If $\sigma$ is quasi-geodesic, then $H$ is finitely generated and quasi-isometrically embedded in $G$.
 \item[(2)] If $\sigma$ is a semihyperbolic structure for $G$, then $H$ is semihyperbolic.
\end{enumerate}
\end{lemma}

In addition, according to \cite[Proposition 7.5]{Alonso-Bridson95}, centralizers are quasi-convex with respect to a semihyperbolic structure for the entire group.

\begin{lemma}{\rm(}{$\mathrm{Short}$}{\rm)}\label{cent_qconvex}
Let $G$ be a finitely generated group with a (possibly not quasi-geodesic but) bounded
equivariant bicombing $\sigma$.
Then, the centralizer of every element in $G$ is $\sigma$-quasi-convex.
\end{lemma}
\begin{proof}[Proof of Lemma~\ref{cent_qconvex}]
Let $a$ be an element of $G$. 
We denote the centralizer of $a$ in $G$ by $Z(a)$. 
Pick an element $g$ of $Z(a)$. 
We wish to find a constant $k \geq 0$ such that $\sigma(g)$ is uniformly $k$-close to $Z(a)$. 
Let $T_g$ be the length of the combing line $\sigma(g)$. 
By the equivariance of $\sigma$, for each $t \leq T_g$, there exists an element $\gamma_t \in G$ such that $\sigma(g)(t) \gamma_t \sigma(g)(t)^{-1} = a$ and that $d_{G}(1, \gamma_t) \leq k_1 d_{G}(1, a) + k_2$, where $k_1$ and
$k_2$ are the constants for a bounded bicombing $\sigma$. 
As $\sigma(g)(t) \gamma_t \sigma(g)(t)^{-1} = a$, there exists a shortest word $\psi_t$ such that $\psi_t \gamma_t \psi_t^{-1} = a$. 
Set 
$$P:= \cup_{g \in Z(a)} \cup_{1 \leq t \leq T_g} \{ \psi_t \}. $$ 
As the set $\{ \gamma_t \mid g \in Z(a), \ 1 \leq t \leq T_g \}$ is contained in the closed ball of radius $k_1 d_{G}(1, a) + k_2$ centred at $1$, another set $P$ is also bounded. 
Therefore, $k: = \mathrm{max} \{d_{G}(1, \psi) \mid \psi \in P \}$ exists. 
The identity $\sigma(g)(t) \psi_t^{-1} a \psi_t \sigma(g)(t)^{-1} = a$ implies that
$\sigma(g)(t) \psi_t^{-1} \in Z(a)$. 
Thus, $Z(a)$ is $\sigma$-quasi-convex with the above constant $k$. 
\end{proof}

A key ingredient of our proof for Theorem~\ref{first_thm} is that the (extended) mapping
class groups of orientable hyperbolic surfaces are semihyperbolic. 

\begin{lemma}{\rm(}\cite[Corollary D]{Durham-Minsky-Sisto20},~\cite[Corollary
3.11]{Haettel-Hoda-Petyt20}{\rm)}\label{mcg_is_semihyp}
For any orientable hyperbolic surface $S$ of finite type, the mapping class group
$\mathrm{Mod}(S)$ and extended mapping class group $\mathrm{Mod}^{\pm}(S)$ of $S$
are semihyperbolic.
Here, $\mathrm{Mod}^{\pm}(S)$ is the group consisting of the isotopy classes of
homeomorphisms on $S$ which preserve $\partial S$ pointwise.
\end{lemma}

\begin{remark}\label{mcg_is_hyp_for_low_complexity}
We can see the semihyperbolicity of the mapping class groups $\mathrm{Mod}(S)$ of the
orientable surface $S$ with non-negative Euler characteristics as follows. 
There are seven orientable surfaces with non-negative Euler characteristics, $S_{0}$,
$S_{0,1}$, $S_{0}^{1}$, $S_{0,2}$, $S_{0}^{2}$, $S_{0,1}^{1}$, and $S_{1}$. 
Only three of these surfaces have non-trivial mapping class groups, namely, 
$\mathrm{Mod}(S_{0,2})\cong\mathbb{Z}_{2}$, 
$\mathrm{Mod}(S_{0}^{2})\cong\mathbb{Z}$, and
$\mathrm{Mod}(S_{1,0})\cong\mathrm{SL}(2,\mathbb{Z})$. 
All these groups are Gromov hyperbolic, and so they are semihyperbolic.
In these cases, $\mathrm{Mod}^{\pm}(S)$ is also Gromov hyperbolic, and thus
semihyperbolic. 
\end{remark}

In conclusion, we have the following lemma. 

\begin{lemma}\label{centralizer_of_mcg_is_undistorted}
Let $S$ be an orientable surface of finite type.
The centralizer of every element in the (extended) mapping class group of $S$ is quasi-
isometrically embedded and semihyperbolic.
\end{lemma}

\section{Nonorientable surface mapping class groups are undistorted}\label{Nonorientable
surface mapping class groups are undistorted}

In this section, we prove Theorem~\ref{first_thm}. 
We first explain the orientation double covering of a nonorientable surface. 
We represent $S_{g-1,2p}^{2b}$ ($g\geq 1$) in the three-dimensional Euclidean space
$\mathbb{R}^{3}$ in such a way that it is invariant under the composition of the three
reflections about the $xy$, $yz$, and $zx$ planes, as illustrated in 
Figure~\ref{fig_orientable_surface_in_R3}. 
\begin{figure}[h]
\includegraphics[scale=0.34]{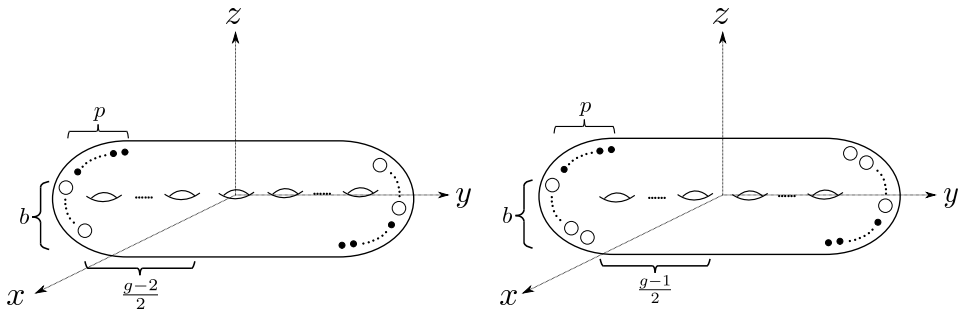}
\caption{We represent $S_{g-1,2p}^{2b}$ in $\mathbb{R}^{3}$ as the left (resp.~right)
surface when $g-1$ is odd (resp.~even), where the $2b$ circles are the boundary components
and the $2p$ points are the punctures of $S_{g-1,
2p}^{2b}$.}\label{fig_orientable_surface_in_R3}
\end{figure}
Then, we define an involution $J\colon S_{g-1,2p}^{2b} \rightarrow S_{g-1,2p}^{2b}$ by
$J(x,y,z)=(-x,-y,-z)$. 
The relation $(x,y,z)\sim J(x,y,z)$ generates an equivalence relation on $S_{g-1,2p}^{2b}$. 
The projection $j\colon S_{g-1,2p}^{2b}\rightarrow S_{g-1,2p}^{2b}/{\sim}$ is a covering
map of order two, and the quotient space $S_{g-1,2p}^{2b}/{\sim}$ is the nonorientable
surface $N_{g,p}^{b}$ (see Figure~\ref{fig_double_cover_v5}). 
In other words, the double covering space of $N_{g, p}^{b}$ induced by the projection $j$ is $S_{g-1, 2p}^{2b}$. 
As mentioned in Section~\ref{Introduction}, Birman--Chillingworth~\cite[Theorem
1]{Birman-Chillingworth72} (for closed surfaces, where $g\geq 3$),
Szepietowski~\cite[Lemma 3]{Szepietowski10} (for surfaces with boundaries, where
$g,b\geq 1$), and Gon\c{c}alves--Guaschi--Maldonado~\cite[Theorem 1.1]{Goncalves-Guaschi-Maldonado18} (for surfaces with punctures, where $g,p\geq 1$) proved that the
mapping class group $\mathrm{Mod}(N_{g, p}^{b})$ is a subgroup of
$\mathrm{Mod}(S_{g-1, 2p}^{2b})$, which consists of the elements commuting with the
isotopy class of $J$. 
Moreover, for nonorientable surfaces $N_{g,p}^{b}$ with $p \geq 1, b \geq 1$, we can obtain an injective homomorphism $\mathrm{Mod}(N_{g,p}^{b})\hookrightarrow\mathrm{Mod}(S_{g-1, 2p}^{2b})$ by applying the same argument used by Szepietowski~\cite[Lemma 3]{Szepietowski10} to the result of Gon\c{c}alves--Guaschi--Maldonado~\cite[Theorem 1.1]{Goncalves-Guaschi-Maldonado18}. 
This can be formulated in the following lemma. 

\begin{lemma}{\rm(}\cite[Theorem 1]{Birman-Chillingworth72},~\cite[Lemma
3]{Szepietowski10},~\cite[Theorem 1.1]{Goncalves-Guaschi-Maldonado18}{\rm)}\label{injective_homomorphism_induced_by_orientation_double_cover} 
For all but $(g, p, b) = (1,0,0),~(2, 0, 0)$, the orientation double covering $j$ induces an injective homomorphism $\iota\colon\mathrm{Mod}(N_{g, p}^{b})\hookrightarrow\mathrm{Mod}(S_{g-1, 2p}^{2b})$. 
Moreover, the image of $\mathrm{Mod}(N_{g,p}^{b})$ given by $\iota$ consists of the
isotopy classes of orientation-preserving homeomorphisms of $S_{g-1,2p}^{2b}$ which
commute with $J$. 
\end{lemma}
\begin{figure}[h]
\includegraphics[scale=0.36]{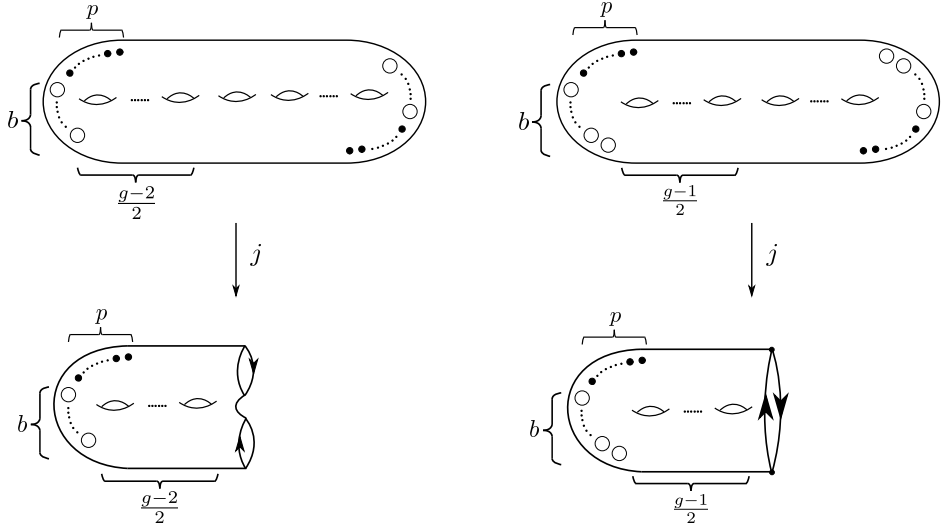}
\caption{The orientation double covering $j\colon S_{g-1, 2p}^{2b}\rightarrow N_{g,
p}^{b}$ when $g-1$ is odd (left) and even (right).}\label{fig_double_cover_v5}
\end{figure}
We can now prove Theorem~\ref{first_thm}. 

\begin{proof}[Proof of Theorem~\ref{first_thm}]
As $\mathrm{Mod}(N_{1})$ is trivial, it is clear that $\mathrm{Mod}(N_{1})$ is quasi-
isometrically embedded in $\mathrm{Mod}(S_{0})$. 
Thus, we may assume that $(g,p,b)$ is neither $(1,0,0)$ nor $(2, 0, 0)$.
Set $S=S_{g-1, 2p}^{2b}$. 
Because the homeomorphism $J$ defined above is orientation-reversing, the isotopy class
$[J]$ is contained in the extended mapping class group $\mathrm{Mod}^{\pm}(S)$. 
Let $N=N_{g,p}^{b}$ be the quotient space $S/{\sim}$ given by the projection $j$. 
By Lemma~\ref{injective_homomorphism_induced_by_orientation_double_cover}, $\mathrm{Mod}(N)$ is a subgroup of the centralizer $Z([J])$ in $\mathrm{Mod}^{\pm}(S)$. 
Moreover, as every orientation-preserving homeomorphism of $S$ which commutes with $J$
preserves the fibres of all points of $N$, all orientation preserving elements in $Z([J])$ come from $\mathrm{Mod}(N)$, and therefore the index of $\mathrm{Mod}(N)$ in $Z([J])$ is two. 
Then, we have the quasi-isometry $f\colon\mathrm{Mod}(N)\hookrightarrow Z([J])$ induced
by $\iota$, because every finite-index subgroup is quasi-isometric to the ambient group. 
By the definition of $f$, we have that $f(\varphi) = \iota(\varphi)$ for any $\varphi \in
\mathrm{Mod}(N)$. 
Note that $\mathrm{Mod}(S)$ is an index-two subgroup of $\mathrm{Mod}^{\pm}(S)$.
Hence, we have the quasi-isometry $h\colon\mathrm{Mod}(S)\rightarrow\mathrm{Mod}^{\pm}(S)$ induced by the inclusion map. 
We define a quasi-inverse $h' \colon\mathrm{Mod}^{\pm}(S)\rightarrow\mathrm{Mod}(S)$
as follows. 
If $\varphi\in\mathrm{Mod}^{\pm}(S)$ is an orientation-preserving element, then
$h' (\varphi)=\varphi\in\mathrm{Mod}(S)$.
If $\varphi\in\mathrm{Mod}^{\pm}(S)$ is an orientation-reversing element, then there exists some $\psi\in\mathrm{Mod}(S)$ such that $\varphi=[J]\psi$ (i.e.~$\psi= [J] \varphi$), and so we set $h'(\varphi)=\psi$. 
The composition $h' \circ h$ then gives the identity map on $\mathrm{Mod}(S)$. 
By Lemma~\ref{centralizer_of_mcg_is_undistorted}, we see that $Z([J])$ is quasi-
isometrically embedded in $\mathrm{Mod}^{\pm}(S)$. 
Let $g\colon Z([J])\hookrightarrow\mathrm{Mod}^{\pm}(S)$ be the quasi-isometric
embedding induced by the inclusion map. 
Then, for any $\varphi\in\mathrm{Mod}(N)$, we have $(g\circ f)(\varphi) = (h \circ
\iota)(\varphi) \in\mathrm{Mod}^{\pm}(S)$. 
Consider the composition of $h'$, $g$, and $f$: 
$$h' \circ g\circ f\colon\mathrm{Mod}(N)\hookrightarrow
Z([J])\hookrightarrow\mathrm{Mod}^{\pm}(S)\rightarrow\mathrm{Mod}(S). $$
We then have that $(h' \circ g\circ f)(\varphi)=\iota(\varphi)$ for any
$\varphi\in\mathrm{Mod}(N)$. 
In other words, $\iota\colon\mathrm{Mod}(N)\hookrightarrow\mathrm{Mod}(S)$ 
decomposes into a composition of three quasi-isometric embeddings, so we are done.
\end{proof}

As shown in the proof of Theorem\ref{first_thm}, $\mathrm{Mod}(N)$ is a finite-index
subgroup of $Z([J])$ in the semihyperbolic group $\mathrm{Mod}^{\pm}(S)$. 
Hence, by Lemmas~\ref{qconvex_qiemb} and~\ref{cent_qconvex}, we have the following
corollary. 

\begin{corollary}\label{nonori_semihyp}
Let $N$ be a nonorientable surface of genus $g\geq 1$ with $b\geq 0$ boundary components
and $p\geq 0$ punctures.
Then, the mapping class group $\mathrm{Mod}(N)$ is semihyperbolic.
\end{corollary}
\begin{proof}[Proof of Corollary~\ref{nonori_semihyp}]
If $N\not=N_{1},~N_{2}$, the semihyperbolicity of $\mathrm{Mod}(N)$ comes from the
fact that it is a finite-index subgroup of the centralizer of $[J]$ in
$\mathrm{Mod}^{\pm}(S)$, and so we only have to prove the assertion for
$N=N_{1},~N_{2}$. 
The mapping class groups satisfy $\mathrm{Mod}(N_{1}) = 1$ and
$\mathrm{Mod}(N_{2})\cong \mathbb{Z}_{2}\oplus\mathbb{Z}_{2}$, respectively. 
They are finite groups, so we are done. 
\end{proof}
\begin{remark}\label{Klein_bottle_case}
We can show that $\mathrm{Mod}(N_{2}) \cong \langle x, y \mid x^2, y^2, [x, y] \rangle$ is
never embedded in $\mathrm{Mod}(S_{1}) = \langle a, b \mid a^4, b^6, a^2b^{-3} \rangle$
as follows. 
Suppose, on the contrary, that there exists an injective homomorphism $\phi \colon
\mathrm{Mod}(N_{2}) \hookrightarrow \mathrm{Mod}(S_{1})$. 
Let $\pi \colon \mathrm{Mod}(S_{1}) \rightarrow \langle \bar{a}, \bar{b} \mid \bar{a}^2,
\bar{b}^3 \rangle = \mathrm{PSL}(2, \mathbb{Z})$ be the canonical projection whose kernel
is generated by $a^2$. 
As $\mathrm{Ker}\pi$ consists of two elements, without loss of generality, we may assume
that $\pi(\phi(x))$ is non-trivial. 
By the Kurosh subgroup theorem, there exists an element $g \in \mathrm{PSL}(2,
\mathbb{Z})$ such that $\pi(\phi(x)) = g \bar{a} g^{-1}$. 
Though $\phi(x)$ must be of order $2$, the elements in the preimage $\pi^{-1}(g \bar{a} g^{-1})$ have order $4$, a contradiction. 
\end{remark}

\section{Appendix}\label{Appendix}
From the work of Masur--Minsky~\cite{Masur-Minsky00} and Hamenst\"{a}dt's unpublished
paper \cite{Hamenstad09b}, the injective homomorphism between the mapping class groups
of orientable surfaces which is induced by an inclusion of the surfaces is a quasi-isometric embedding. 
In this appendix, we prove a generalization of this result (Proposition~\ref{qie_nonori_inc}) by using the semihyperbolicity of the mapping class groups. 
In the following, we do not consider surfaces of infinite type; thus, we assume that any
surface has finite genus and finite numbers of boundary components and punctures.

Let $F$ be a connected surface.
We say that a subsurface $F' \subset F$ is {\it admissible} if $F'$ is a closed subset of $F$. 
For an admissible subsurface $F' \subset F$, we have a homomorphism $\mathrm{Mod}(F')
\rightarrow \mathrm{Mod}(F)$ by extending the homeomorphisms of $F'$ to the
homeomorphisms of $F$ which are trivial on the outside of $F'$. 
Paris--Rolfsen~\cite{Paris-Rolfsen00} and Stukow~\cite{Stukow09} proved that, under the assumption in Proposition~\ref{qie_nonori_inc}, this natural homomorphism
$\mathrm{Mod}(F')\rightarrow\mathrm{Mod}(F)$ is injective. 

\begin{proposition}\label{qie_nonori_inc}
Let $F$ be a connected orientable or nonorientable surface and $F' \subset F$ be an
admissible connected subsurface.
Suppose that every connected component of $F-\mathrm{Int}(F')$ has a negative Euler
characteristic. 
Then, the injective homomorphism $\mathrm{Mod}(F') \hookrightarrow \mathrm{Mod}(F)$
is a quasi-isometric embedding.
\end{proposition}

Proposition~\ref{qie_nonori_inc} can be reduced to the following lemma.

\begin{lemma}\label{qie_2}
Let $F$ be a connected orientable or nonorientable surface and $F' \subset F$ be an
admissible connected subsurface. 
Suppose that every connected component of $F-\mathrm{Int}(F')$ has a negative Euler
characteristic. 
Then, there exists a finite-index subgroup $H$ of $\mathrm{Mod}(F)$ such that the natural
injection $\mathrm{Mod}(F') \cap H \hookrightarrow H$ is a quasi-isometric embedding. 
\end{lemma}

To prove Lemma \ref{qie_2}, we prepare the following lemmas.

\begin{lemma}\label{filling_curves}
Let $F$ be a connected orientable or nonorientable surface of genus $g$ with $b \geq 1$
boundary components and $p$ punctures. 
We assume that $b+p \geq 4$ if $F$ is orientable and $g=0$. 
We also assume that $g+b+p \geq 4$ if $F$ is nonorientable. 
Then, there exists a pair $\{\alpha_{1}, \alpha_{2}\}$ of essential simple closed curves
satisfying the following properties. 
\begin{enumerate}
\item[(1)] If $F$ is nonorientable, then the closed curves $\alpha_{1}, \alpha_{2}$ are two-sided. 
\item[(2)] $F-(\mathrm{Int}N(\alpha_1) \cup \mathrm{Int}N(\alpha_2) )$ is a disjoint union
of some copies of $N_{1}^{1}$, $S_{0}^{1}$, $S_{0, 1}^{1}$, and $S_{0}^{2}$.
\end{enumerate}
\end{lemma}
\begin{proof}[Proof of Lemma~\ref{filling_curves}]
Suppose that $F$ is orientable.
Then, the curve complex of $F$ has infinite diameter (see Masur--Minsky~\cite[Theorem
1.1]{Masur-Minsky99}). 
This implies that $F$ has a pair $\{\alpha_{1}, \alpha_{2}\}$ of essential simple closed
curves satisfying condition (2) in Lemma~\ref{filling_curves}. 
We can reduce the case where $F$ is nonorientable to the case where $F$ is orientable by
replacing some of the punctures with crosscaps. 
Then, the pair of closed curves do not pass through a crosscap and each closed curve is two-sided, thereby satisfying condition (1). 
\end{proof}

Let $F$ be a surface.
A closed curve $\beta$ on $F$ is called {\it peripheral} if $\beta$ is isotopic to a component of $\partial F$. 
A two-sided closed curve $\alpha$ on $F$ is called {\it generic} if $\alpha$ bounds neither a disk nor a M\"{o}bius strip and is not peripheral.
Let $\mathcal{T}(F)$ denote the subgroup of $\mathrm{Mod}(F)$, called the {\it twist
subgroup}, generated by Dehn twists along two-sided closed curves which are either
peripheral or generic on $F$. 

\begin{lemma}\label{twist_subgp}
We have the following.
\begin{enumerate}
\item[(1)] $\mathcal{T}(N_{1,1}^{1}) \cong \mathbb{Z}$, and its generator is a Dehn twist
along a unique peripheral closed curve.
\item[(2)] $\mathcal{T}(N_{1}^{2}) \cong \mathbb{Z}^2$, and its generators are Dehn
twists along peripheral closed curves.
\item[(3)] $\mathcal{T}(N_{2}^{1}) \cong \mathbb{Z}^2$, and its generators are a Dehn
twist along a unique peripheral closed curve and a Dehn twist along a unique generic closed curve on $N_{2}^{1}$.
\end{enumerate}
\end{lemma}
\begin{proof}[Proof of Lemma~\ref{twist_subgp}]
According to \cite[Propositions 17]{Paris14}, $\mathrm{Mod}(N_{1,1}^{1}) \cong
\mathbb{Z}$ and is generated by a ``boundary slide" $s$. 
As the square of $s$ is isotopic to a Dehn twist along a unique peripheral closed curve on
$N_{1,1}^{1}$, the twist subgroup $\mathcal{T}(N_{1,1}^{1})$ is generated by a Dehn twist. 

To obtain an isomorphism $\mathbb{Z}^2 \rightarrow \mathcal{T}(N_{1}^{2})$, we use the
capping homomorphism $\mathrm{Mod}(N_{1}^{2}) \rightarrow \mathrm{Mod}(N_{1,1}^{1})$ induced by gluing $N_{1}^{2}$ with a punctured disk along a boundary component $C$ of $N_{1}^{2}$. 
Then, the kernel of the capping homomorphism is generated by a Dehn twist along a closed
curve isotopic to $C$. 
Additionally, the image of a Dehn twist along a peripheral closed curve on $N_{1}^{2}$
which is not isotopic to $C$ is $s^2$. 
Hence, $\mathcal{T}(N_{1,0}^{2})$ is freely generated by Dehn twists along those
peripheral closed curves. 
By \cite[Propositions 22]{Paris14}, we have $\mathrm{Mod}(N_{2}^{1}) \cong \mathbb{Z} \rtimes \mathbb{Z}$. 
In addition, the first copy of $\mathbb{Z}$ is generated by a Dehn twist along a unique
generic closed curve on $N_{2}^{1}$ and the second copy is generated by a ``crosscap slide" $y$. 
As the square of $y$ is isotopic to a Dehn twist along a peripheral closed curve on
$N_{2,0}^{1}$, $\mathcal{T}(N_{2}^{1})$ is freely generated by those Dehn twists. 
\end{proof}

The next lemma asserts that the mapping class group of any ``essential" subsurface, excepting a few examples, is virtually isomorphic to a direct factor of a $\sigma$-quasi-convex subgroup of the ambient mapping class group. 

\begin{lemma}\label{realization_subsurface_mcg_for_centralizer}
Let $F$ be a connected orientable or nonorientable surface and $F' \subset F$ be an
admissible connected subsurface which is not an annulus. 
Suppose that $\mathrm{Mod}(F') \neq 1$ and that every connected component of $F-
\mathrm{Int}(F')$ has a negative Euler characteristic. 
Then, there exist mapping classes $\varphi_1, \ldots, \varphi_l \in \mathcal{T}(F) $ such that a finite-index subgroup of $\cap_{i=1}^{l}Z_{\mathcal{T}(F)}(\varphi_i)$ is isomorphic to $\mathcal{T}(F') \times \mathbb{Z}^{r}$. 
\end{lemma}
Here, $Z_{\mathcal{T}(F)}(\varphi_i)$ is the centralizer of $\varphi_i$ in $\mathcal{T}(F)$ and the index $r$ in Lemma~\ref{realization_subsurface_mcg_for_centralizer} is the sum of the number of boundary components of $F$ which are not contained in $F'$ and the number of connected components of $F-\mathrm{Int}(F')$ which are homeomorphic to a one-holed
Klein bottle. 

\begin{proof}[Proof of Lemma~\ref{realization_subsurface_mcg_for_centralizer}]
Let $F_1, \ldots, F_n$ be the connected components of $F-\mathrm{Int}F'$. 
We denote the genus of $F_i$, the number of boundary components of $F_i$, and the number
of punctures of $F_i$ as $g(F_i)$, $b(F_i)$, and $p(F_i)$, respectively. 
As the Euler characteristic of $F_i$ is negative, $F_i$ satisfies exactly one of the following conditions: 
\begin{enumerate}
\item[(a)] $F_i$ is orientable and either $g(F_i) \geq 1$ or $b(F_i)+p(F_i) \geq 4$.
\item[(b)] $F_i$ is orientable, $g(F_i)=0$, and $b(F_i)+p(F_i) = 3$.
\item[(c)] $F_i$ is nonorientable and $g(F_i) + b(F_i)+p(F_i) \geq 4$.
\item[(d)] $F_i$ is nonorientable and $g(F_i) + b(F_i)+p(F_i) = 3$.
\end{enumerate}
If $F_i$ satisfies condition (a) or (c), we have a pair $P_i$ of essential closed curves which fills $F_i$ in the sense of Lemma~\ref{filling_curves}. 
We define a set of closed curves $A_i$ to be a union of $P$ and the set of closed curves of $F_i$ which are parallel to $\partial F'$. 
In the case where $F_i$ satisfies condition (b) or (d), the set $A_i$ is defined to be the set of closed curves of $F_i$ which are parallel to $\partial F'$. 
Set $\varphi_{\alpha}:= [T_{\alpha}]$ for each $\alpha \in A:= \cup_{i=1}^{n}A_i$ and set
$$ B_{i}:=
\begin{cases}
\langle [T_{\beta}] \mid \beta \in \partial F \cap \partial F_i \rangle & (\mbox{if} \ F_i \not
\cong N_{2}^{1}) \\
\langle [T_{\gamma}] \mid \gamma \ \mbox{is a two sided generic closed curve} \rangle & 
(\mbox{if} \ F_i \cong N_{2}^{1}) \end{cases}
. $$
Note that $B_i \cong \mathbb{Z}$ when $F_i \cong N_{2}^{1}$. 
Consider a subgroup $(\mathrm{Mod}(F')B_1 \cdots B_n) \cap \mathcal{T}(F)$ of
$\mathcal{T}(F)$. 
As $F'$ is not an annulus and $\mathrm{Mod}(F') \neq 1$, for each component $C$ of
$\partial F' \cap \partial F_i$, there exists a two-sided essential closed curve $\gamma_{C}$ in $F$ such that $\gamma_{C}$ intersects $C$ non-trivially in minimal position and is disjoint from $(\partial F \cup \partial F') - \{ C \}$ (and a unique two-sided generic closed curve on $F_i$ if $F_i \cong N_{2}^{1}$). 
As all elements in $B_i$ are commutative with $[T_{\gamma_C}]$, we have $\mathrm{Mod}(F') \cap B_i = 1$. 
Therefore, $(\mathrm{Mod}(F') B_1 \cdots B_n ) \cap \mathcal{T}(F) = \mathcal{T}(F') B_1 \cdots B_n \cong \mathcal{T}(F') \times \mathbb{Z}^{r}$, where $r$ is the sum of the free abelian rank of $B_1, \ldots, B_n$ and is equal to the sum of the number of boundary components of $F$ which are not
contained in $F'$ and the number of connected components of $F-\mathrm{Int}F'$ which are
homeomorphic to $N_{2}^{1}$. 
In addition, it clearly holds that
$$(\mathrm{Mod}(F') B_1 \cdots B_n) \cap \mathcal{T}(F) \subset \cap_{\alpha \in A}
Z_{\mathcal{T}(F)}(\varphi_{\alpha}) .$$
To simplify the notation, we denote $\cap_{\alpha \in A} Z_{\mathcal{T}(F)}(\varphi_{\alpha})$ by $Z$. 

We now claim that $(\mathrm{Mod}(F')B_1 \cdots B_n) \cap \mathcal{T}(F)$ is a finite-
index subgroup of $Z$. 
To see this, consider a subset $\mathcal{S}$ of $Z$ realizing all possible reversing patterns on orientations of closed curves in $A$. 
If there is no element of $Z$ which reverses an orientation of a closed curve in $A$, we set $\mathcal{S}= \{ 1 \}$. 
As $A$ is finite, we can choose $\mathcal{S}$ to be finite. 
Pick an element $f$ in $Z$. 
Then, $f$ preserves each closed curve in $A$, and so there exists an element $s \in
\mathcal{S}$ such that $sf$ fixes an orientation of each closed curves in $A$. 
In the following, we prove that $sf \in (\mathrm{Mod}(F')B_1 \cdots B_n) \cap
\mathcal{T}(F)$. 
This immediately implies that $(\mathrm{Mod}(F')B_1 \cdots B_n) \cap \mathcal{T}(F)$ is
of finite index in $Z$. 
As $sf$ fixes an orientation of each closed curve in $A$, $sf$ can be decomposed as a
product of mapping classes of the regular neighbourhood $N(A)$ of $A$ and $F - \mathrm{Int}N(A)$. 
By Lemma~\ref{filling_curves}, $F - \mathrm{Int}N(A)$ is a disjoint union of $F'$, outer
surfaces $F_i$ satisfying condition (b) or (d), and some copies of $S_{0}^{1}$, $S_{0
,1}^{1}$, $S_{0}^{2}$, $N_{1}^{1}$. 
Obviously, $sf |_{F'}$ is contained in $\mathrm{Mod}(F')$. 
Additionally, if $F_i$ satisfies condition (b) or (d), we have that $sf |_{F_i}$ is contained in $\mathrm{Mod}(F') B_i$ by Lemma~\ref{twist_subgp} and the fact that
$\mathrm{Mod}(F_i)$ is an abelian group freely generated by Dehn twists along peripheral
closed curves if $F_i$ satisfies condition (b). 
Note that the copies of $S_{0}^{2}$ are in one-to-one correspondence with the components
of $\cup_{i=1}^{n} \partial F_i$. 
Hence, the restriction of $sf$ to the copies of $S_{0}^{1}$, $S_{0 ,1}^{1}$, $S_{0}^{2}$,
and $N_{1}^{1}$ in $F - \mathrm{Int}N(A)$ is contained in $\mathrm{Mod}(N(\cup_{i=1}^{n} \partial F_i)) \subset \mathrm{Mod}(F')B_1 \cdots B_n$ by Alexander's theorem and Epstein's theorem~\cite[Proposition 5]{Paris14}. 
Therefore, we have that $sf |_{F - \mathrm{Int}N(A)} \in \mathrm{Mod}(F')B_1 \cdots
B_n$. 
Furthermore, we can verify that $sf |_{N(A)}$ is contained in $\mathrm{Mod}(F')$. 
To see this, we use the fact that $sf$ and $sf |_{F - \mathrm{Int}N(A)}$ are commutative
with all of $\varphi_{\alpha}$ ($\alpha \in A$). 
As $sf |_{N(A)} = sf \cdot (sf |_{F - \mathrm{Int}N(A)})^{-1}$, the restriction $sf |_{N(A)}$ is also commutative with all of $\varphi_{\alpha}$ ($\alpha \in A$). 
If $F_i$ satisfies condition (b) or (d), the restriction of $sf |_{N(A)}$ to $F_i$ is contained in $\mathrm{Mod}(F')$, because $A_i \subset F'$. 
If $F_i$ satisfies condition (a) or (c), the restriction of $sf |_{N(A)}$ to $F_i$ is contained in $\mathrm{Mod}(F')$, because $sf |_{N(A)}$ should be trivial on the regular neighbourhood of the filling pair $P_i$.
Therefore, $sf |_{N(A)} \in \mathrm{Mod}(F')$, and so $sf \in \mathrm{Mod}(F')B_1 \cdots
B_n$. 
Because $sf \in \mathcal{T}(F)$, we have that $sf \in (\mathrm{Mod}(F')B_1 \cdots B_n)
\cap \mathcal{T}(F)$, as desired. 
\end{proof}

\begin{remark}
Assume that $F=S_{0,p+2}$ and $F'=S_{0,p}^{1}$ in Lemma~\ref{realization_subsurface_mcg_for_centralizer}. 
If $p \geq 3$, we can conclude that the braid
group on $p$-strands coincides with $Z([T_{\beta}])$ in $\mathrm{Mod}(S_{0,p+2})$
because $\mathcal{S} = \{ 1 \}$. 
Here, $\beta$ is a peripheral closed curve parallel to $\partial F'$. 
\end{remark}

We are now ready to prove Lemma~\ref{qie_2}.
Recall that the mapping class group of an orientable or nonorientable surface is
semihyperbolic, and the intersection of two quasi-convex subgroups is also quasi-convex with
respect to a semihyperbolic structure (cf. Bridson--Haefliger~\cite[Proposition 4.13, Chapter
I\hspace{-1pt}I\hspace{-1pt}I.$\Gamma$]{Bridson-Haefliger99}). 

\begin{proof}[Proof of Lemma~\ref{qie_2}]
We first consider the case where $\mathrm{Mod}(F')=1$. 
In this case, Proposition~\ref{qie_nonori_inc} is trivial. 
We next consider the case where $F'$ is an annulus. 
Then, Proposition~\ref{qie_nonori_inc} can be obtained by using the semihyperbolicity of
$\mathrm{Mod}(F)$, because any finitely generated abelian subgroup is quasi-isometrically
embedded in a semihyperbolic group (see Bridson--Haefliger~\cite[Theorem 4.10, Chapter
I\hspace{-1pt}I\hspace{-1pt}I.$\Gamma$]{Bridson-Haefliger99}). 
We now assume that $\mathrm{Mod}(F') \neq 1$ and $F'$ is not an annulus. 
By Lemma~\ref{realization_subsurface_mcg_for_centralizer}, there exist mapping classes
$\varphi_1, \ldots, \varphi_l \in \mathcal{T}(F)$ and a non-negative number $r$ such that
$\mathcal{T}(F') \times \mathbb{Z}^{r}$ is naturally embedded in $\cap_{i=1}^{l}
Z_{\mathcal{T}(F)}(\varphi_i)$ as a finite-index subgroup. 
Lickorish~\cite{Lickorish63} proved that $\mathcal{T}(F)$ is a finite-index subgroup of
$\mathrm{Mod}(F)$ if $F$ is closed. 
Because $F$ is either closed or an admissible subsurface of some closed nonorientable
surface, Lickorish's theorem together with Paris--Rolfsen~\cite{Paris-Rolfsen00} and
Stukow~\cite{Stukow09} implies that $\mathcal{T}(F)$ is a finite-index subgroup of
$\mathrm{Mod}(F)$. 
Let $\sigma$ be a semihyperbolic structure of $\mathcal{T}(F)$.
As each direct factor is quasi-isometrically embedded in a given direct product, the subgroup $\mathcal{T}(F')$ is quasi-isometrically embedded in $\cap_{i=1}^{l}
Z_{\mathcal{T}(F)}(\varphi_i)$. 
We then obtain the result that the inclusion map from $\mathcal{T}(F')$ to $\mathcal{T}(F)$ is a quasi-isometric embedding by the $\sigma$-quasi-convexity of $\cap_{i=1}^{l}
Z_{\mathcal{T}(F)}(\varphi_i)$. 
\end{proof}

Finally, we remark that for closed surfaces, hyperelliptic mapping class groups are also undistorted subgroups because they are centralizers of the mapping class groups (see Stukow~\cite{Stukow15} for the definition of hyperelliptic mapping class groups of closed nonorientable surfaces).

\end{document}